\documentclass[10pt,fleqn]{article}
\usepackage{amsmath}
\usepackage{amssymb}
\usepackage{amsthm}
\usepackage{graphicx}
\usepackage{morefloats}
\usepackage{mathrsfs}
\usepackage{dsfont}
\usepackage{natbib}

\graphicspath{{images/}}
\newtheorem{theorem}{Theorem}

\newtheorem{lemma}{Lemma}

\providecommand{\abs}[1]{\lvert#1\rvert}
\providecommand{\norm}[1]{\lVert#1\rVert}

\usepackage{fullpage}
\begin{document}

\title{Convergence of a Metropolized Integrator for Stochastic Differential Equations with Variable Diffusion Coefficient}
\author{Paul Tupper, Xin Yang\footnote{Author for correspondence (xiny@sfu.ca).}
\\ \\ Department of Mathematics, Simon Fraser University, 8888 University Drive\\
 Burnaby British Columbia, Canada V5A 1S6}
\maketitle

\section*{Abstract}
We present an explicit method for simulating stochastic differential equations (SDEs) that have variable diffusion coefficients and satisfy the detailed balance condition with respect to a known equilibrium density.
In \cite{paulxin} we proposed a framework for such systems in which, instead of a diffusion coefficient and a drift coefficient, a modeller specifies a diffusion coefficient and an equilibrium density, and then assumes detailed balance with respect to this equilibrium density. We proposed a numerical method for such systems that works directly with the diffusion coefficient and equilibrium density, rather than the drift coefficient, and uses a Metropolis-Hastings rejection process to preserve the equilibrium density exactly.
 Here we show that the method is weakly convergent with order $1/2$ for such systems with smooth coefficients.  We perform numerical experiments demonstrating the convergence of the method for systems not covered by our theorem, including systems with discontinuous diffusion coefficients and equilibrium densities.

\section{Introduction}
Consider a system of It\^{o} stochastic differential equations of the form
\begin{equation} \label{eqn:sde0}
dX=a(X)dt+b(X)dB_t,
\end{equation}
where $a$ is a vector function of $X$, $b$ is a scalar function of $X$, and $B_t$ is standard $d$-dimensional Brownian motion.
 Letting $D(X) = \frac{1}{2} b^2(X)$, the Fokker-Planck equation for this system is
\[
\frac{\partial \rho(x,t)}{\partial t}=-\nabla\cdot [a(x)\rho(x,t)]+ \Delta[D(x)\rho(x,t)]=-\nabla \cdot J,
\]
where we have defined the density flow
\[
J(x,t):=a(x)\rho(x,t)-\nabla[D(x)\rho(x,t)].
\]
If there is a density $\rho_{eq}$ such that
\[
J(x)=a(x)\rho_{eq}(x)-\nabla[D(x)\rho_{eq}(x)]=0
\]
for all $x$,
then we say the system \emph{satisfies detailed balance with respect to $\rho_{eq}$}.
A direct result of the condition is that $\rho_{eq}$ is an equilibrium density of the system. In closed, isolated physical systems, the solution satisfying the detailed balance condition is known as the thermal equilibrium distribution \cite{vankampen}.
Diffusions that satisfy the detailed balance condition with respect to some invariant measure feature prominently in many areas of physics, chemistry, and mathematical biology \cite{vankampen,gardiner,bourabee2014}.   
From the perspective of stochastic differential equations, the usual way of modelling such systems is to specify the coefficients $a(X)$ and $b(X)$ and let equation \eqref{eqn:sde0} describe the evolution of the system in time.

Instead of starting with $a(X)$ and $b(X)$, in \cite{paulxin} we proposed that the modeller specifies $D(X)$ and $\rho_{eq}(X)$ and assumes detailed balance. These assumptions are enough to uniquely determine the coefficients $a$ and $b$ of \eqref{eqn:sde0}:  we get
\begin{equation} \label{coefdefs}
a(X)= \nabla D(X)+D(X)\nabla \ln
\rho_{eq}(X), \ \ \ \ \
b(X) = \sqrt{2D(X)}.
\end{equation}
Thus the SDE  \eqref{eqn:sde0} takes the form
\begin{equation}
dX(t)=[\nabla D(X)+D(X)\nabla \ln
\rho_{eq}(X)]dt+\sqrt{2D(X)}dB_t\\
\label{sde1}
\end{equation}
with  Fokker-Planck equation
\begin{equation}
\frac{\partial \rho(x,t)}{\partial t}= \nabla \cdot \left[ D(x) \rho_{eq}(x) \nabla \left( \frac{ \rho(x,t)} {\rho_{eq}(x)}\right)\right].
\label{pde:fpe}
\end{equation}
The advantages of this change of perspective are two-fold: (i) in many circumstances it is more natural to model the system in terms of $\rho_{eq}$ and $D$, such as when $\rho_{eq}$ is available from experimental data but $a$ is not \cite{siggia},  (ii) there are situations in which $D$ and $\rho_{eq}$ are well-defined but $a$ is singular, such as when $D$ or $\rho_{eq}$ has a jump discontinuity \cite{paulxin}. In this case, defining algorithms in terms of $D$ and $\rho_{eq}$ allows us to avoid working with a singular drift $a$.

%
%
%

Since the equilibrium distribution $\rho_{eq}$ plays an important role in the system \eqref{sde1}, we want a numerical method that has an identical equilibrium distribution. As shown by \cite{roberts1996}, the standard Euler-Maruyama scheme (EM) or Milstein's method will typically not preserve the correct equilibrium density. (In fact, due to instability, these methods may not be ergodic at all even when the underlying diffusion is exponentially ergodic).
Roberts and Tweedie introduced the Metropolis-adjusted Langevin algorithm (MALA) as a way of simulating the system while  keeping the exact equilibrium distribution. Their method proposes a trial step using the Euler-Maruyama scheme  and then decides whether to accept or reject the trial step using the Metropolis-Hastings procedure with the correct known value of $\rho_{eq}$. \cite{bourabee2010} have shown that MALA is not only ergodic with respect to $\rho_{eq}$ but also converges to the solution of the SDE strongly. Our method is a variant of the MALA scheme. Instead of using a convergent scheme for the SDE, we only use the diffusion coefficient to give a trial step and then use the Metropolis-Hastings rejection procedure to guarantee the correct equilibrium density. Therefore the drift is enforced only indirectly through the rejection step. The motivation for this idea is that for any SDE the infinitesimal drift  is uniquely determined by the infinitesimal diffusion, the equilibrium distribution, and the detailed balance condition \cite{paulxin}.  Therefore, if we have a Markov chain that approximates a diffusion process with the correct diffusion coefficient and the correct equilibrium distribution, and  also satisfies the detailed balance condition,  we expect that the process also has  approximately the correct drift coefficient.
For our scheme, since the trial step is given with the correct diffusion and the Metropolis-Hastings rejection process provides the detailed balance with respect to the correct equilibrium density, we expect that it converges to the correct solution of the stochastic differential equation.  In this paper, we will show directly that the process has the correct drift and diffusion in the limit of steplength going to zero, when the coefficients are sufficiently smooth. In particular, we show that the scheme is weakly convergent with order of accuracy $1/2$ under appropriate conditions.

A similar theorem appears in \cite{bourabee2014} for general self-adjoint diffusions for a class of Metropolized integrators that includes ours as a special case. Their method consists of the use of a Runge-Kutta type integrator for the trial step followed by a Metropolis-Hastings decision to accept or reject the step. In general, their trial step uses the gradient of the diffusion coefficient, but also allows our choice of using only the diffusion coefficient itself as a special case (corresponding to $G_h = 0$ in their notation). Their more general framework also includes the possibility of using the gradient of $\ln \rho_{eq}$ to obtain a more accurate trial step.

Though our results here are for smooth coefficients, the main motivation for our scheme is to to handle instances of \eqref{sde1} where $D$ has jump discontinuities.  Other work has developed numerical schemes for similar classes of problems. The reference \cite{labolle} proposes a method for such systems that  does not make explicit use of the equilibrium distribution and hence does not preserve it exactly. However, their method could be adjusted with a Metropolis-Hastings step in order to do so. Another approach is to resolve the jump discontinuities in $D$ by developing a separate procedure for when the state of the system approaches the discontinuity. This approach is taken by \cite{etore,lejay2012,martinez_talay} for one-dimensional systems, who make use of the theory of skew Brownian motion to resolve the discontinuity.

Here we define our algorithm  from \cite{paulxin} for approximating the solution of \eqref{sde1}. Let $h$ be the step length. 
The trial step is given by
\begin{eqnarray}
X^{*}_{n+1}=X_n+\sqrt{2D(X_n)}[B((n+1)h)-B(nh)].
\label{trialstep}
\end{eqnarray}
This is accepted with probability $\alpha_h$
\begin{eqnarray}
X_{n+1}=
\left\{
\begin{array}{rl}
X^*_{n+1}, & \text{if } \xi_{n}<\alpha_h\left(X_n,X^*_{n+1}\right),\\
X_n,   &  \text{otherwise.}\\
\end{array}
\right.
\label{mhrejection}
\end{eqnarray}
where $\xi_{k}$ satisfies uniform distribution on [0,1] and
$\alpha_h$ is the acceptance probability for Metroplis-Hastings rejection procedure \cite{roberts1996}
from state $X_n$ to $X^*_{n+1}$ with the expression
\begin{equation}
\alpha_h(x,y)=\min \left(1,\frac{q_h(y,x)\rho_{eq}(y)}{q_h(x,y)\rho_{eq}(x)} \right)
\label{alpha}
\end{equation}
and $q_h(x,y)$ is the transitional probability density determining the trial step \eqref{trialstep}
\begin{equation}
q_h(x,y)=\frac{1}{(4\pi h
D(x))^{d/2}}e^{-\frac{(x-y)^2}{4hD(x)}}.
\label{qh}
\end{equation}
This choice of $\alpha_h$ and $q_h$ in the Metropolis-Hastings rejection process guarantees that the process $X_{n}, n = 0,1,2, \ldots$ satisfies detailed balance with respect to the density $\rho_{eq}$.

\section{Weak convergence of the method}

Firstly, we exhibit some sufficient conditions on $D$ and $\rho_{eq}$ for the ergodicity of the SDE \eqref{sde1} and the numerical scheme in Theorem \ref{pdeergodic} and Theorem \ref{ergodicity}. The convergence to the equilibium $\rho_{eq}$ of \eqref{pde:fpe} is shown using the idea of relative entropy and the logarithmic Sobolev inequality \cite{arnold2001}. As we show in Theorem 2, the numerical method is ergodic and has the correct equilibrium distribution because of the use of the Metropolis-Hastings method. We then show in Theorem 3 that the numerical method converges weakly with order $1/2$ for smooth $\rho_{eq}$ and $D$. \\      
We will let $H(\rho_1|\rho_2)$ be the relative entropy of $\rho_1$ with respect to $\rho_2$ where
\[
H(\rho_1|\rho_2):=\int_{\mathbb{R}^d}\rho_1(x) \ln \frac{\rho_1(x)}{\rho_2(x)}dx.
\]
The reason to use relative entropy is due to Csisz\`ar-Kullback inequality
\begin{equation}
H(\rho_1|\rho_2) \geq \frac{1}{2} \norm{\rho_1-\rho_2}^2_{L^1} 
\label{ck}
\end{equation}
Therefore, once we have convergence in the relative entropy, we will have convergence in $L^1$.
Another useful functional $I(\rho_1|\rho_2)$ called entropy dissipation functional is defined by
\[
I(\rho_1|\rho_2):=\int_{\mathbb{R}^d}\rho_1(x) \nabla \ln \frac{\rho_1(x)}{\rho_2(x)} \cdot \nabla \ln \frac{\rho_1(x)}{\rho_2(x)}dx.
\]
\begin{theorem}
Suppose
\begin{enumerate}
\item The known equilibrium density $\rho_{eq} \in C^2(\mathbb{R}^d)$ is positive $\rho_{eq}(x)>0$ and satisfies $\nabla^2 \ln \rho_{eq} \leq -\lambda I_d$, where $I_d$ is the identity matrix of dimension $d$ and $\lambda>0$ is some positive constant.
\item $H(\rho(x,0)| \rho_{eq}(x))<\infty$. i.e. the initial condition of \eqref{pde:fpe} has finite relative entropy with respect to the equilibrium density $\rho_{eq}$.
\item The diffusion coefficient $D \in C^2(\mathbb{R}^d)$ and $D$ is bounded below by some positive number: $\inf D(x) = D_{\min}>0$.
\item The surface integral  
\[
\int_{\abs{x}=R} D \rho_{eq} \left| \nabla \frac{\rho}{\rho_{eq}} \right| dx 
\]
vanishes as $R \rightarrow +\infty$.  
\end{enumerate}
then $\rho(x,t)$ converges to $\rho_{eq}$ exponentially fast in relative entropy.
\[
H(\rho(x,t)|\rho_{eq}) \leq e^{-2t\lambda D_{\min}}H(\rho(x,0)|\rho_{eq})
\]
Hence, $\rho(x,t)\rightarrow \rho_{eq}(x)$ in $L^1$ as $t\rightarrow \infty$.
\label{pdeergodic}
\end{theorem}
\begin{proof}
Let $g=\frac{\rho}{\rho_{eq}}$. Assuming $\rho$ is a solution to \eqref{pde:fpe}, $g$ will satisfy
\[
\frac{\partial g}{\partial t}=\frac{\nabla \cdot \left( D \rho_{eq} \nabla g \right)}{\rho_{eq} }.
\] 
Let $\phi(g)=g\ln g -g+1$, then through direct calculation, $H(\rho|\rho_{eq})=\int_{\mathbb{R}^d} \phi(g) \rho_{eq}dx$ and 
\[
\frac{d}{dt} H(\rho|\rho_{eq})= \int_{\mathbb{R}^d} \frac {\partial \phi(g)}{\partial g} \frac{\partial g}{\partial t}\rho_{eq}dx=- \int_{\mathbb{R}^d} D \rho \left( \nabla \ln \frac{\rho}{\rho_{eq}} \cdot \nabla \ln \frac{\rho}{\rho_{eq}} \right) dx \leq - D_{\min} \cdot I(\rho(x,t)|\rho_{eq}(x))
\]
where the surface integral from integration by parts vanishes because of condition 3.  
By Theorem 1 in \cite{markowich}, condition 1 here guarantees that the logarithmic Sobolev inequality with parameter $\lambda$ holds
\[
H(\rho|\rho_{eq}) \leq \frac{1}{2\lambda} I(\rho | \rho_{eq}).
\]
As a result,
\[
\frac{d}{dt} H(\rho|\rho_{eq}) \leq -2\lambda D_{\min} H(\rho|\rho_{eq})
\]
We get the exponential convergence in relative entropy which will imply exponential convergence in $L^1$ by \eqref{ck}.
\end{proof}

{\bf Remark}: Theorem \ref{pdeergodic} also works when $\rho_{eq}$ is only positive in some connected open set $\mathcal{D}\subset \mathbb{R}^d$ provided that the condition 4 is replaced by zero-flux boundary conditions on $\partial \mathcal{D}$. By restricting the domain inside the region, $\rho_{eq}$ will be strictly positive inside the domain and there's no problem of dividing by zero. A discussion about relaxing the uniform convexity of $\nabla^2 \ln \rho_{eq}$ in condition 1 could be found in \cite{markowich}. 

\begin{theorem}
Suppose the diffusion coefficient $D$ is bounded below by some positive number: $\inf D (x)>0$ and suppose $\nu$ is the equilibrium probability distribution with density $\rho_{eq}$. Let the numerical scheme defined in \eqref{trialstep}, \eqref{mhrejection} generate a Markov chain with n-step transitional probability distribution $P^n(x,\cdot)$.  Then $P^n(x,\cdot)$ converges to the equilibrium probability distribution $\nu(\cdot)$ in total variation norm as $n\rightarrow \infty$ i.e. :
\[
\sup \{ \left| P^n(x,A) -\nu(A)\right| : \mbox{for all } \mbox{measurable set } A \}   \rightarrow 0\mbox{ , uniformly in x} 
\]
\label{ergodicity}
\end{theorem}
\begin{proof}
The proof follows from \cite{jarner2000,smith1993}. We only need to show that the chain generated by the numerical method is $\rho_{eq}$-irreducible and aperiodic. These two conditions are satisfied since 1) our proposal step is given by Gaussian random variables which gives a positive probability to any set with positive Lebesgue measure, 2) the acceptance rate $\alpha_h(x,y)$ in Metropolis-Hastings rejection step will always be positive as long as $\rho_{eq}(y)$ is positive. Hence the transitional distribution of the Markov chain with rejections generated by the numerical method will have a positive probability of jumping into any set where $\rho_{eq}$ is positive. 
\end{proof}

Now we show the main result of this paper that the scheme is weakly convergent. We rewrite the time stepping of the scheme in the form
\[
X_{n+1}=X_n+\bar A(X_n,h;X^*_{n+1},\xi_{n})
\]
where
\[
\bar{A}(X_n,h;X^*_{n+1},\xi_{n})=(X^*_{n+1}-X_n) \mathds{1}_{\xi_{n}<\alpha_h\left(X_n,X^*_{n+1}\right)}
\]
is the increment of the numerical scheme in a single step. We shall use $A$ to denote the increment of the exact solution in a single step.

\begin{theorem}[Weak convergence of the scheme]
Suppose that
\begin{enumerate}
\item The diffusion coefficient $D$ and the logarithm of the equilibrium density $\ln \rho_{eq}(x)$ have bounded derivatives up to order $4$.
\item $\norm{\nabla^2 D(x)}$ and $\norm{\nabla^2 \ln \rho_{eq}(x)}$ can be bounded by some polynomial in $x$ and the diffusion coefficient $D(x)$ is bounded away from zero: $\inf(D(x)) >0$.
\item the function $f(x)$ together with its partial derivatives of order up to and including $3$ have at most polynomial growth.
\end{enumerate}
We assume the initial condition $X(0)=x_0$ is fixed. For uniform discretization $t_k=hk$, $k=1,...,N$ with $t_N=T$ the total time, the following inequality holds for all $k$:
\[
\abs{\mathbb{E}(f(X(t_k))-f(X_k))}\leq C h^{1/2}.
\]
\label{mainthm}
\end{theorem}
{\bf Remark:} The first condition in Theorem \ref{mainthm} is made in terms of $D$ and $\rho_{eq}$ to fit the framework of SDE in \eqref{sde1}. The same result can be obtained under a weaker condition if one impose the smoothness in terms of the coefficients $a$ and $b$ in \eqref{eqn:sde0}, i.e., when the coefficients $a(x)=\nabla D(x) + D(x) \nabla \ln \rho_{eq}(x)$ and $b(x) = \sqrt{2 D(x)}$ of equation \eqref{sde1} are continuous, satisfy a Lipschitz condition
\[
\abs{a(x)-a(y)}+\abs{b(x)-b(y)}\leq K \abs{x-y}
\]
and together with their partial derivatives with respect to $x$ of order up to and including $3$ have at most polynomial growth.

We prove Theorem~\ref{mainthm} by analyzing the local error of the scheme.
In the following estimates of the local error, we use the same techniques as \cite{bourabee2014},  making  precise the  dependence of the remainder term on $x$ in order to guarantee global convergence. We also have a slightly less restrictive condition on $\rho_{eq}(x)$ in that derivatives of $\ln \rho_{eq}(x)$ do not need to be bounded.
\begin{proof}
We are going to apply Theorem 2.1 of \cite[p. 100]{milstein} to show the weak convergence of the scheme. The condition (a)  of their theorem corresponds to our condition  1  which is the requirement on the smoothness and the growth of the coefficients $a$ and $b$.  The condition (c) there corresponds to our condition 3  which is the requirement on the smoothness and the growth of the test function $f$. Their condition (d)  is a uniform \emph{a priori} bound  on the moments of the numerical scheme which is  guaranteed by our Lemma \ref{momentsbound}.
What remains to be shown is their condition (b): bounds on the moments of the increments of the numerical method. For convenience, we use $\mathcal{O}(x,h^p)$ to denote a quantity that can be bounded by $K_1(x)h^{p}$ where $K_1(x)$ is some polynomial or a matrix of polynomial entries.\\
The condition (b) in Theorem 2.1 of \cite[p. 100]{milstein} has two requirements. Firstly, all the third moments of the increment in the numerical scheme must be $\mathcal{O}(X_n,h^{3/2})$, i.e.
\[
\mathbb{E}_{X_n}\left(\prod_{j=1}^{3}\abs{\bar A^{i_j}}\right) \leq K_1(X_n) h^{3/2}
\]
Here $\bar A^{i_j}$ is the $i_j \in \{1,..,d\} $ component of $\bar A$ and $K_1(x)$ is a function with at most polynomial growth.
Then, the difference between the first and second moments of the approximated increment and the exact increment needs to be $\mathcal{O}(X_n,h^{3/2})$, i.e.
\[
\left| \mathbb{E}_{X_n}\left(\prod_{j=1}^{s} \bar A^{i_j} - \prod_{j=1}^{s} A^{i_j}\right) \right| \leq K_2(X_n) h^{3/2}, \ \ \ s=1,2
\]
Here $K_2(x)$ is also a function with at most polynomial growth.\\
For the first requirement, since
\[
\left|\bar{A}(X_n,h;X^*_{n+1},\xi_{n})\right|=\left|(X^*_{n+1}-X_n) \mathds{1}_{\xi_{n}<\alpha_h\left(X_n,X^*_{n+1}\right)}\right| \leq  \left| \sqrt{2D(X_n)}\left(B^{i_j}((n+1)h)-B^{i_j}(nh))\right) \right|
\]
therefore
\[
E_{X_n}\left(\prod_{j=1}^{3}\abs{\bar A^{i_j}}\right)\leq  (2D(X_n))^{3/2} h^{3/2}
\]
By the Lipschitz condition on $b(x)=\sqrt{2D(x)}$,  $(2D(X_n))^{3/2}$ will be bounded by some polynomial.
For the second requirement, consider the solution after one time step from the initial condition. Let $A(X(0),h)=X(h)-X(0) $ be a column vector of the increment of the exact solution.
\[
E_{X_0}(\bar A(X_0,h;X^*_{1},\xi_{1}) - A(X_0,h)) =\mathbb{E}_{X_0}((X_1^*-X_0)\alpha_h(X_0,X^*_1)-(X(h)-X_0))
\]
\[
E_{X_0}(\bar A(X_0,h;X^*_{1},\xi_{1})\bar A^T(X_0,h;X^*_{1}) - A(X_0,h)A^{T}(X_0,h))\]
\[
=\mathbb{E}_{X_0}((X_1^*-X_0)(X_1^*-X_0)^T\alpha_h(X_0,X^*_1)-(X(h)-X_0)(X(h)-X_0)^T)
\]

By Theorem \ref{conv}, we have
\[
\mathbb{E}_{X_0}((X_1^*-X_0)\alpha_h(X_0,X^*_1)) =a(X_0)h + \mathcal{O}(X_0,h^{3/2})\\
\]
\[
\mathbb{E}_{X_0}(X_1^*-X_0)(X_1^*-X_0)^T\alpha_h(X_0,X^*_1)=b^2(X_0)h I_d + \mathcal{O}(X_0,h^{3/2})\\
\]
Let $\mathcal{L}f(x)=a^{T}(x) \cdot \nabla_x f(x) + \frac{1}{2} b(x) \Delta_x f(x)$ be the infinitesimal generator of the It\^{o} diffusion \eqref{sde1}. By Ito-Taylor expansion \cite[p.99]{milstein} , we have the expansion componentwise
\begin{eqnarray}
\mathbb{E}_{X_0}(X(h)-X_0)^{i} =a^{i}(X_0)h +\mathbb{E}_{X_0} \left(h\int_{0}^{h} \mathcal{L}a^{i}(X(t))dt\right)
\label{exactremainder}
\end{eqnarray}
\begin{eqnarray}
&&\mathbb{E}_{X_0}((X(h)-X_0)(X(h)-X_0)^T)^{ij}=b^2(X_0)h I^{ij}_d+\nonumber \\
&&\mathbb{E}_{X_0} \left(h\int_{0}^{h} \mathcal{L} \left(a^{i}(X(t))\cdot (X^{j}(t)-X^j(0)) + a^{j}(X(t))\cdot(X^{i}(t)-X^{i}(0))+ \frac{1}{2}b^{2}(X(t))\right) dt\right)
\label{exactremainder22}
\end{eqnarray}
Since the integrands in the remainder terms in \eqref{exactremainder}  \eqref{exactremainder22} are combinations of products of $X,a,b$ and their derivatives, by assumptions on their growth, the integrands can only have at most polynomial growth in $X$. We can find $m$ large enough, s.t.
\[
\left|\mathbb{E}_{X_0} \left(h\int_{0}^{h} \mathcal{L}a^{i}(X(t))dt\right)\right|<h\mathbb{E}_{X_0}  \int_{0}^{h} C_1(1+\abs{X(t)}^{2m})dt\\
\]
for some constant $C_1$. The Theorem 4 in \cite[p. 48]{gihman}  shows that the moments of the solution could be uniformly bounded by the moments of the initial condtion, i.e.
\[
\mathbb{E}_{X_0}  \int_{0}^{h} C_1(1+\abs{X(t)}^{2m})dt \leq  h\mathbb{E}_{X_0}  C(1+\abs{X(0)}^{2m}) = h C(1+\abs{X(0)}^{2m})
\]
The constant $C$ in the last inequality only depends on $T$, $m$, $K$. The same process applies to the remainder in \eqref{exactremainder22}.
As a result, \eqref{exactremainder} \eqref{exactremainder22} becomes,
\begin{eqnarray}
\mathbb{E}_{X_0}(X(h)-X_0)^{i} &=&a^{i}(X_0)h + \mathcal{O}(X(0),h^2)\\
\mathbb{E}_{X_0}((X(h)-X_0)(X(h)-X_0)^T)^{ij}&=&b^2(X_0)h I^{ij}_d + \mathcal{O}(X(0),h^2) 
\label{exactremainder2}
\end{eqnarray}
Hence, we have the weak local error,
\[
\abs{E_{X_0}(\bar A(X_0,h;X^*_{1},\xi_{1}) - A(X_0,h))}\leq   \mathcal{O}(X(0),h^{3/2})
\]
\[
\abs{E_{X_0}(\bar A(X_0,h;X^*_{1},\xi_{1})\bar A^T(X_0,h;X^*_{1}) - A(X_0,h)A^{T}(X_0,h)) }\leq  \mathcal{O}(X(0),h^{3/2})
\]

Therefore, by Theorem 2.1 in \cite[p. 100]{milstein}, the method is convergent with order of accuracy $1/2$.
\end{proof}

\begin{lemma}
Suppose the assumptions in Theorem \ref{mainthm} are satisfied. Then for every even number $2m$ the $2m$-moment of the numerical solution $\mathbb{E}\abs{X_k}^{2m}$ exist and are uniformly bounded with respect to $k=1,...,N$, if and only if $\mathbb{E}\abs{X_0}^{2m}$ exists.
\label{momentsbound}
\end{lemma}
\begin{proof}
The result follows from Lemma 2.2 in \cite{milstein}, if the magnitude of $\bar A$ in one step is well-behaved. By using Theorem \ref{conv}, the expectation of $A$ is of order $h$
\begin{eqnarray*}
&&\abs{\mathbb{E}_{X_n}\bar A(X_{n},h;X^*_{n+1},\xi_{n+1})}=\left|\mathbb{E}_{X_n}\left((X^*_{n+1}-X_n) \mathds{1}_{\xi_{n}<\alpha_h(X_n,X^*_{n+1})}\right)\right|\\
&&=\abs{\mathbb{E}_{X_n}\left((X^*_{n+1}-X_n) \alpha_h(X_n,X^*_{n+1})\right)} \leq  K(1+\abs{X_n})h
\end{eqnarray*}
while $\abs{\bar A}$ is of order $h^{1/2}$
\[
\abs{\bar A(X_{n},h;X^*_{n+1},\xi_{n+1})}\leq \abs{X^*_{n+1}-X_n} \leq  \left|\frac{X^*_{n+1}-X_n}{\sqrt{2D(X_n)h}} \right| \sqrt{2D(X_n)h}
\]
and $\frac{X^*_{n+1}-X_n}{\sqrt{2D(X_n)h}} $ satisfies the standard normal distribution and hence has moments of all orders. Then by Lemma 2.2 in \cite{milstein}, the moments of the numerical solution $\mathbb{E}\abs{X_k}^{2m}$ exist and are uniformly bounded.
\end{proof}

\begin{theorem}
With the definitions and assumptions in Theorem \ref{mainthm}, we have the following,
\begin{eqnarray*}
&&\mathbb{E}_{X_0}(X_1^*-X_0)\alpha_h(X_0,X^*_1)=a(X_0)h + \mathcal{O}(X_0,h^{3/2})\\
&&\mathbb{E}_{X_0}(X_1^*-X_0)(X_1^*-X_0)^T\alpha_h(X_0,X^*_1)=b^2(X_0)h I_d + \mathcal{O}(X_0, h^{3/2})
\end{eqnarray*}
\label{conv}
\end{theorem}
\begin{proof}
For convenience, let $x=X_0$, $y=X_1^*$, and we can rewrite the conditional expectation in the integral form,
\begin{equation}
\mathbb{E}_{X_0}(X_1^*-X_0)\alpha_h(X_0,X^*_1)=\int_{\mathbb{R}^d}(y-x)\alpha_h(x,y)
q_h(x,y)dy
\label{drifteq}
\end{equation}
\begin{equation}
\mathbb{E}_{X_0}(X_1^*-X_0)(X_1^*-X_0)^T\alpha_h(X_0,X^*_1)=\int_{\mathbb{R}^d}(y-x)(y-x)^T\alpha_h(x,y)\cdot
q_h(x,y)dy
\label{diffeq}
\end{equation}
Introducing a change of variable, let $\epsilon=\sqrt{h}$, $y-x=\epsilon z$. Therefore the transition probability density changes into
\[
q_h(x,y)dy=\frac{1}{(\sqrt{4\pi h
D(x)})^d}e^{-\frac{(x-y)^2}{4hD(x)}}dy=\frac{1}{(\sqrt{4\pi
D(x)})^d}e^{-\frac{z^2}{4D(x)}}dz=:q(x,z)dz
\]
which is independent of  $\epsilon$.
Let

\[
\alpha(x,z,\epsilon)=\min\left(1,\frac{q(x+\epsilon z,z)\rho_{eq}(x+\epsilon z)}{q(x,z)\rho_{eq}(x)}\right)
\]
After the change of variable, \eqref{drifteq} and \eqref{diffeq} become,
\begin{equation}
\int_{\mathbb{R}^d}(y-x)\alpha_h(x,y)\cdot q_h(x,y)dy=\epsilon\int_{\mathbb{R}^d}z\alpha(x,z,\epsilon)q(x,z)dz
\label{drifteq2}
\end{equation}
\begin{equation}
\int_{\mathbb{R}^d}(y-x)(y-x)^T(\alpha_h(x,y))\cdot q_h(x,y)dy=\epsilon^2 \int_{\mathbb{R}^d}zz^T(\alpha(x,z,\epsilon))\cdot q(x,z)dz
\label{diffeq2}
\end{equation}
Let
\[
\beta(x,z,\epsilon)=\min \left(1,\exp\left(\epsilon \frac{\nabla_x q(x,z)\cdot z}{q(x,z)}+\epsilon \frac{\nabla_x \rho_{eq}(x) \cdot z}{\rho_{eq}(x)}\right)\right)
\]
be an approximation for $\alpha(x,z,\epsilon)$. The motivation of $\beta$ is discussed in Lemma \ref{alphaestimate}.
First we study the order of the error in drift. Applying the fact that $\int_{\mathbb{R}^d}zq(x,z)dz=0$ which follows from the symmetry of $q$, we obtain
\[
\epsilon\int_{\mathbb{R}^d}z\alpha(x,z,\epsilon)q(x,z)dz=\epsilon \int_{\mathbb{R}^d}z(\alpha(x,z,\epsilon)-1)q(x,z)dz.
\]
By Lemma \ref{alphaestimate} and Lemma \ref{driftcoeff}, we can obtain
\begin{eqnarray*}
&&\epsilon \int_{\mathbb{R}^d}z(\alpha(x,z,\epsilon)-1)q(x,z)dz=\epsilon \int_{\mathbb{R}^d}z(\beta(x,z,\epsilon)-1)q(x,z)dz +\epsilon \int_{\mathbb{R}^d}z (\alpha(x,z,\epsilon)-\beta(x,z,\epsilon))q(x,z)dz\\
&&=a(x)\epsilon^2 + \mathcal{O}(x,\epsilon^3)
\end{eqnarray*}
Use Lemma \ref{alphaestimate} and Lemma \ref{driftcoeff} for \eqref{diffeq2},
\begin{eqnarray*}
&&\epsilon^2 \int_{\mathbb{R}^d}zz^T(\alpha(x,z,\epsilon))\cdot q(x,z)dz=\\
&&\epsilon^2 \int_{\mathbb{R}^d}zz^T\cdot q(x,z)dz+\epsilon^2 \int_{\mathbb{R}^d}zz^T(\beta(x,z,\epsilon)-1)\cdot q(x,z)dz+\epsilon^2 \int_{\mathbb{R}^d}zz^T(\alpha(x,z,\epsilon)-\beta(x,z,\epsilon)) \cdot q(x,z)dz\\
&&=b(x) I^d \epsilon^2+ \mathcal{O}(x,\epsilon^3)\\
\end{eqnarray*}
Recall that $\epsilon=\sqrt{h}$, therefore we have the desired bounds for local error.
\end{proof}

\begin{lemma}[Estimates of $\alpha(x,z,\epsilon)$ and $\beta(x,z,\epsilon)$]

With previous definitions, we have the following estimates.
Let $g(z)\in \mathbb{R}$ be polynomial in $z$, then
\[
\left| \int_{\mathbb{R}^d}g(z)(\alpha(x,z,\epsilon)-\beta(x,z,\epsilon))q(x,z)dz \right| \leq K(x)\epsilon^2
\]
where $K(x)$ has polynomial growth.
\label{alphaestimate}

\end{lemma}

\begin{proof}
Rewrite $\alpha$ in exponent form,
\[
\alpha(x,z,\epsilon)=\min\left(1,\exp\left( \ln \frac{q(x+\epsilon z,z)\rho_{eq}(x+\epsilon z)}{q(x,z)\rho_{eq}(x)} \right) \right)
\]
A Taylor expansion for the exponent about $\epsilon=0$ gives
\[
\frac{q(x+\epsilon z,z)\rho_{eq}(x+\epsilon z)}{q(x,z)\rho_{eq}(x)}=\exp\left(\epsilon \frac{\nabla_x q(x,z)}{q(x,z)}+\epsilon \frac{\nabla_x \rho_{eq}(x)}{\rho_{eq}(x)} + R(x,z,\epsilon) \right)
\]
Therefore $\beta$ is obtained by keeping only the leading order $\epsilon$ terms.
\[
\beta(x,z,\epsilon)=\min\left(1,\exp\left(\epsilon \frac{\nabla_x q(x,z)}{q(x,z)}+\epsilon \frac{\nabla_x \rho_{eq}(x)}{\rho_{eq}(x)}  \right)\right)
\]
$R(x,z,\epsilon)$ is the remainder given by
\begin{eqnarray*}
&&R(x,z,\epsilon)=\int_0^{\epsilon} \int_0^{\xi} \frac{\partial^2 \ln  \left( q(x+\eta z,z)\rho_{eq}(x+\eta z)\right)}{\partial \eta^2} d\eta d\xi\\
&&=\int_0^{\epsilon} \int_0^{\xi}   \left( -\frac{d}{2} \frac{(z^T   \nabla^2_w D(w)  z)}{D(w)} + \frac{d}{2} \frac{(z^T  \nabla_w D(w) )^2}{D^2(w)}+ \frac{z^2(z^T  \nabla^2_w D(w)  z) }{4D^2(w)}-\frac{z^2}{8D^3(w)}(z^T \cdot  \nabla_w D(w) )^2    \right)_{w=x+\eta z}    d\eta d\xi\\
&& +\int_0^{\epsilon} \int_0^{\xi} \left(  \frac{z^{T} \cdot \nabla^2_w \rho_{eq}(w) \cdot z}{\rho_{eq}(w)} - \frac{(z^{T} \cdot \nabla_w \rho_{eq}(w))^2}{\rho^2_{eq}(w)} \right)_{w=x+\epsilon z}  d\eta d\xi\\
\end{eqnarray*}
Consider the function $h(x)=\min(1,\exp(x))$. Since $h(x)$ is piecewise smooth, it is not hard to see that $h(x)$ is globally Lipschitz with Lipschitz constant $1$. Therefore,
\[
\abs{\alpha(x,z,\epsilon)-\beta(x,z,\epsilon)}\leq \abs{R(x,z,\epsilon)}
\]
Therefore, with the assumptions that $\inf D(x) >0$ , $\norm{\nabla^2 D(x)}$ bounded by some polynomial, $\norm{\nabla^2 \ln \rho_{eq}(x)}$ bounded by some polynomial, we obtain
\[
\abs{R(x,z,\epsilon)}\leq K_1(x,z) \epsilon^2
\]
Here $K_1(x,z)$ is some polynomial in $x,z$. Furthermore, since for fixed $x$, $q(x,z)$ is a multivariate Gaussian, we can calculate its absolute moments  \cite[p. 337]{table},
\begin{eqnarray*}
\int_{\mathbb{R}^d} \abs{z}^p q(x,z)dz=\left\{
\begin{array}{cc}
\frac{S^d}{2} (2D(x))^{p/2} (p-1)!! & \text{if $p$ is even}\\
\sqrt{\frac{2}{\pi}}\frac{S^d}{2} (2D(x))^{p/2} (p-1)!! & \text{if $p$ is odd}\\
\end{array}
\right.
\end{eqnarray*}
where $S^d$ is the surface area of the unit hypersphere in $\mathbb{R}^d$. Since $b(x)=\sqrt{2D(x)}$ has at most polynomial growth, therefore,
\[
\left| \int_{\mathbb{R}^d} g(z) (\alpha(x,z,\epsilon)-\beta(x,z,\epsilon))q(x,z)dz \right| \leq K(x)\epsilon^2
\]
For $K(x)$ has at most polynomial growth.
\end{proof}
\begin{lemma}
With previous definitions,
\begin{eqnarray}
&&\left|\int_{\mathbb{R}^d}zq(x,z) (\beta(x,z,\epsilon)-1) dz-a(x)\epsilon \right| \leq K(x)\epsilon^2\\
&&\left|\int_{\mathbb{R}^d}g(z)q(x,z) (\beta(x,z,\epsilon)-1) dz\right| \leq K^*(x)\epsilon
\label{l3eq2}
\end{eqnarray}

where $K(x)$,$K^*(x)$ are polynomials in $x$
\label{driftcoeff}
\end{lemma}
\begin{proof}
Since
\begin{eqnarray*}
\beta(x,z,\epsilon)-1=\min \left(0,\exp\left(\epsilon \frac{\nabla_x q(x,z)\cdot z}{q(x,z)}+\epsilon \frac{\nabla_x \rho_{eq}(x) \cdot z}{\rho_{eq}(x)}\right)-1\right)\\
= \left(\exp\left(\epsilon \frac{\nabla_x q(x,z)\cdot z}{q(x,z)}+\epsilon \frac{\nabla_x \rho_{eq}(x) \cdot z}{\rho_{eq}(x)}\right)-1\right)\mathds{1}_{\Omega_{(x,0)}}
\end{eqnarray*}

Here the region is defined by,
\[
\Omega_{(x,0)}=\left\{z\left|\left(\exp\left(\epsilon \frac{\nabla_x q(x,z)\cdot z}{q(x,z)}+\epsilon \frac{\nabla_x \rho_{eq}(x) \cdot z}{\rho_{eq}(x)}\right)-1\right)<0\right\}\right.=\left\{z\left|\left( \frac{\nabla_x q(x,z)\cdot z}{q(x,z)}+ \frac{\nabla_x \rho_{eq}(x) \cdot z}{\rho_{eq}(x)}\right)<0\right\}\right.
\]
Therefore we can expand $\beta(x,z,\epsilon)-1$ in the integrand in domain $\Omega_{(x,0)}$ about $\epsilon=0$
\begin{eqnarray*}
&&\int_{\mathbb{R}^d}zq(x,z)(\beta(x,z,\epsilon)-1)dz=\int_{\Omega_{(x,0)}} zq(x,z)\left(\exp\left(\epsilon \frac{\nabla_x q(x,z)\cdot z}{q(x,z)}+\epsilon \frac{\nabla_x \rho_{eq}(x) \cdot z}{\rho_{eq}(x)}\right)-1\right)dz\\
&&=\int_{\Omega_{(x,0)}} zq(x,z)\left( \epsilon \frac{\nabla_x q(x,z)\cdot z}{q(x,z)}+\epsilon \frac{\nabla_x \rho_{eq}(x) \cdot z}{\rho_{eq}(x)}  + \epsilon^2 R(x,z,\xi(\epsilon))  \right )dz
\end{eqnarray*}
where $\epsilon^2 R$ is the remainder given by,
\[
R(x,z,\xi(\epsilon))=\exp\left(\xi \frac{\nabla_x q(x,z)\cdot z}{q(x,z)}+\xi \frac{\nabla_x \rho_{eq}(x) \cdot z}{\rho_{eq}(x)}\right)\left(  \frac{\nabla_x q(x,z)\cdot z}{q(x,z)}+ \frac{\nabla_x \rho_{eq}(x) \cdot z}{\rho_{eq}(x)}\right)^2
\]
with $0<\xi(\epsilon)<\epsilon$. Notice that
\[
zq(x,z)\left( \epsilon \frac{\nabla_x q(x,z)\cdot z}{q(x,z)}+\epsilon \frac{\nabla_x \rho_{eq}(x) \cdot z}{\rho_{eq}(x)} \right)
\]
is a odd function in $z$. On the other hand the integral domain $\Omega_{(x,0)}$ is odd. Hence the integral without $\epsilon^2R$ term becomes
\begin{eqnarray*}
&&\int_{\Omega_{(x,0)}} zq(x,z)\left( \epsilon \frac{\nabla_x q(x,z)\cdot z}{q(x,z)}+\epsilon \frac{\nabla_x \rho_{eq}(x) \cdot z}{\rho_{eq}(x)} \right )dz=
\frac{\epsilon}{2}\int_{\mathbb{R}^d}zq(x,z)\left(\epsilon \frac{\nabla_x q(x,z)\cdot z}{q(x,z)}+\epsilon \frac{\nabla_x \rho_{eq}(x) \cdot z}{\rho_{eq}(x)}\right)dz\\
&&=\frac{\epsilon}{2}\int_{\mathbb{R}^d}z(\nabla \ln(\rho_{eq}(x)) \cdot z -\frac{d}{2} \nabla \ln(D(x)) \cdot z -\frac{z^T z}{4} \nabla (\frac{1}{D(x)}) \cdot z)q(x,z)dz\\
&&=\left(\nabla D(x)+D(x)\nabla \ln
\rho_{eq}(x)\right) \epsilon=a(x)\epsilon
\end{eqnarray*}
Then we need to show the remainder term is indeed of order $\epsilon^2$. Since in the domain $\Omega_{(x,0)}$,  $\frac{\nabla_x q(x,z)\cdot z}{q(x,z)}+ \frac{\nabla_x \rho_{eq}(x) \cdot z}{\rho_{eq}(x)}<0$, therefore $\exp\left(\xi \frac{\nabla_x q(x,z)\cdot z}{q(x,z)}+\xi \frac{\nabla_x \rho_{eq}(x) \cdot z}{\rho_{eq}(x)}\right)<1$.
\begin{eqnarray*}
&&\left|\int_{\Omega_{(x,0)}} zq(x,z) R(x,z,\xi(\epsilon)) dz \right| =\\
&&\left|\int_{\Omega_{(x,0)}} zq(x,z)\left( \exp\left(\xi \frac{\nabla_x q(x,z)\cdot z}{q(x,z)}+\xi \frac{\nabla_x \rho_{eq}(x) \cdot z}{\rho_{eq}(x)}\right)\left(  \frac{\nabla_x q(x,z)\cdot z}{q(x,z)}+ \frac{\nabla_x \rho_{eq}(x) \cdot z}{\rho_{eq}(x)}\right)^2  \right )dz\right|\leq \\
&&\int_{\Omega_{(x,0)}} \abs{z}q(x,z)\left(  \frac{\nabla_x q(x,z)\cdot z}{q(x,z)}+ \frac{\nabla_x \rho_{eq}(x) \cdot z}{\rho_{eq}(x)}\right)^2 dz\leq\int_{\mathbb{R}^d} \abs{z}q(x,z)\left(  \frac{\nabla_x q(x,z)\cdot z}{q(x,z)}+ \frac{\nabla_x \rho_{eq}(x) \cdot z}{\rho_{eq}(x)}\right)^2 dz
\end{eqnarray*}
As shown in Lemma \ref{alphaestimate}, the term $ \abs{\frac{\nabla_x q(x,z)\cdot z}{q(x,z)}+ \frac{\nabla_x \rho_{eq}(x) \cdot z}{\rho_{eq}(x)}}$ could be bounded by a polynomial $K_1(x,z)$ and since the integral $\int_{\mathbb{R}^d} \abs{z}^p q(x,z)dz$ could be bounded by a polynomial $K_2(x)$, therefore there exists a polynomial $K(x)$
\[
\left|\int_{\Omega_{(x,0)}} zq(x,z) R(x,z,\xi(\epsilon)) dz \right| \leq K(x)
\]
A similar proof works for the other inequality \eqref{l3eq2}. By Taylor expansion,
\begin{eqnarray*}
&&\left|\int_{\mathbb{R}^d}g(z)q(x,z) (\beta(x,z,\epsilon)-1) dz\right|=\\
&&\epsilon \left|\int_{\Omega_{(x,0)}} g(z)q(x,z) \exp\left(\xi \frac{\nabla_x q(x,z)\cdot z}{q(x,z)}+\xi \frac{\nabla_x \rho_{eq}(x) \cdot z}{\rho_{eq}(x)}\right)\left(  \frac{\nabla_x q(x,z)\cdot z}{q(x,z)}+ \frac{\nabla_x \rho_{eq}(x) \cdot z}{\rho_{eq}(x)}\right) dz\right| \leq\\
&&\epsilon \int_{\mathbb{R}^d} \abs{g(z)}q(x,z) \left|  \frac{\nabla_x q(x,z)\cdot z}{q(x,z)}+ \frac{\nabla_x \rho_{eq}(x) \cdot z}{\rho_{eq}(x)}\right| dz\leq K^*(x)\epsilon
\end{eqnarray*}

which concludes the proof.
\end{proof}

\section{Numerical Simulations}
In this section we validate our method with the following numerical experiments. We chose 1-dimensional examples of \eqref{sde1} with  the following features:
\begin{enumerate}
\item Smooth diffusion coefficient $D$ and equilibrium density $\rho_{eq}$, for which we have an exact solution.
\item Smooth and periodic diffusion coefficient $D$ and equilibrium distribution $\rho_{eq}=1$.
\item Geometric brownian motion, for which we have a  degenerate $D$.
\item Piecewise constant $D$ and  $\rho_{eq}$.
\end{enumerate}
The motivation of these examples is to demonstrate the convergence of the numerical scheme for some problems not necessarily satisfying the conditions of Theorem \ref{mainthm}.
\subsection{Example 1: SDE with smooth coefficients. }
We first test the method on a SDE for which we have a closed-form solution,
\begin{eqnarray}
dX=-\frac{X}{2}dt+\sqrt{1-X^2}dB.
\label{eqex1}
\end{eqnarray}
Comparing with \eqref{sde1}, we can see that this is the case when the diffusion coefficient is
\[D(x)=\frac{1-x^2}{2}\]
and the equilibrium density is
 \[\rho_{eq}(x)=\frac{1}{\pi \sqrt{1-x^2}}\]
in the domain $\abs{x}<1$.
If the initial condition is $X(0)=\frac{1}{2}$, then it has the exact solution
\[
X(t)=\sin(B(t)+\frac{\pi}{6}).
\]
The equation \eqref{eqex1} does not satisfy the conditions of Theorem \ref{mainthm} because $D$ is not bounded away from zero and $\frac{d}{dx}\ln(\rho_{eq}(x))$ approaches infinity at $x=\pm1$.\\
Firstly, we numerically verify that this method keeps the exact equilibrium density $\rho_{eq}$ and approximates the given diffusion coefficient. To compute these statistical quantities of the trajectories, the domain $(-1,1)$ is cut into 20 equally spaced subintervals $[x_i,x_{i+1}), i=0,\ldots,19$. The density is computed by dividing the number of times that the particle is in the particular interval over the total number of timesteps. The effective diffusion coefficient is computed as in \cite{paulxin}:
\[
D(x_i)=\underset{X_{kh} \in [x_i,x_{i+1})}{\operatorname{mean}} \frac{(X_{(k+1)h}-X_{kh})^2}{2h}
\]
The SDE is simulated with different timestep lengths over a total time interval of length $T=1000$. With these parameters we plot the values of $\rho_{eq}(x)$ and $D(x)$ over the domain $\abs{x}<1$ in Figure \ref{rhodsinbt}. The error bars show estimates of standard error due to the finite time simulation. As we can see from Figure \ref{rhodsinbt},  the numerical method produces the correct distribution for all the timestep lengths, while the effective $D$ is converging to the exact curve as the time step length is decreasing.
\begin{figure}
\includegraphics[width=1\textwidth]{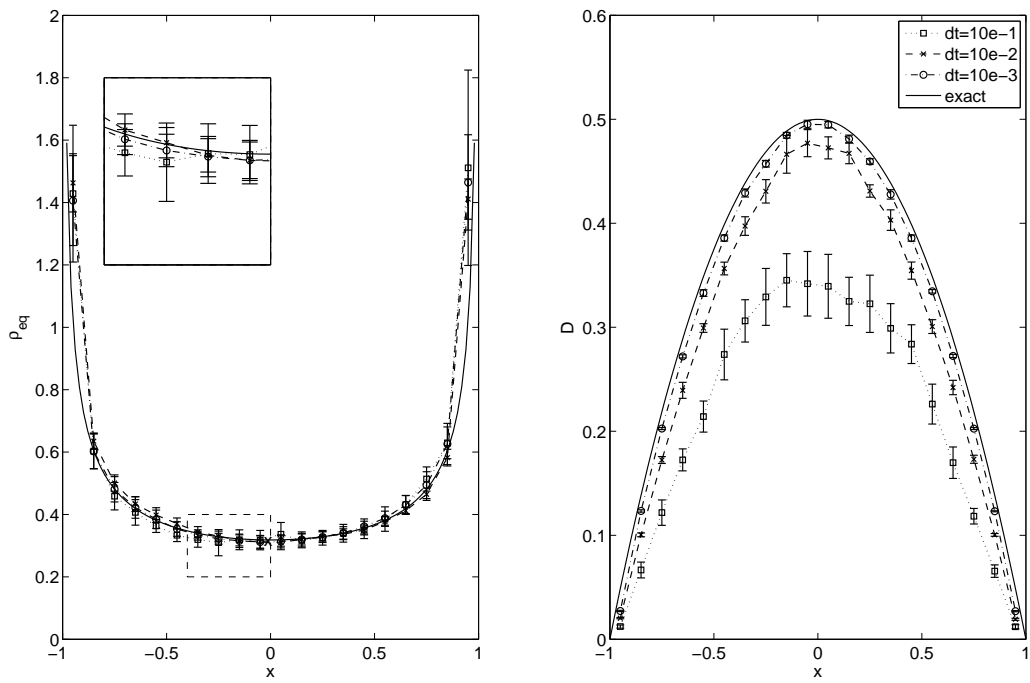}\\

\caption{Plot of the computed equilibrium density (left) and diffusion coefficient (right) for Example 1.}
\label{rhodsinbt}
\end{figure}

In order to check the weak accuracy of the numerical scheme, we measure the mean error at time $T=1$ with test function $f(x)$ as in \cite{higham2001},
\begin{equation}
\epsilon_{h}=\abs{\mathbb{E}(f(X_{Nh}))-\mathbb{E}(f(X(T)))}
\label{errexact}
\end{equation}
The expectation $\mathbb{E}(f(X_{Nh}))$ is approximated by the average values of $f(X_{Nh})$ over a number of $M=10^7$ trajectories.
Figure \ref{sinbt} shows the error versus the time step length with test functions $f(x)=x$ and $f(x)=x^2$. For these test functions, the exact solutions are $\mathbb{E}X(1)=\frac{1}{2\sqrt{e}}$, $\mathbb{E}(X(1))^2=\frac{1}{2}-\frac{1}{4e^2}$ The plot shows the accuracy is of order $1/2$.
\begin{figure}
\includegraphics[width=1\textwidth]{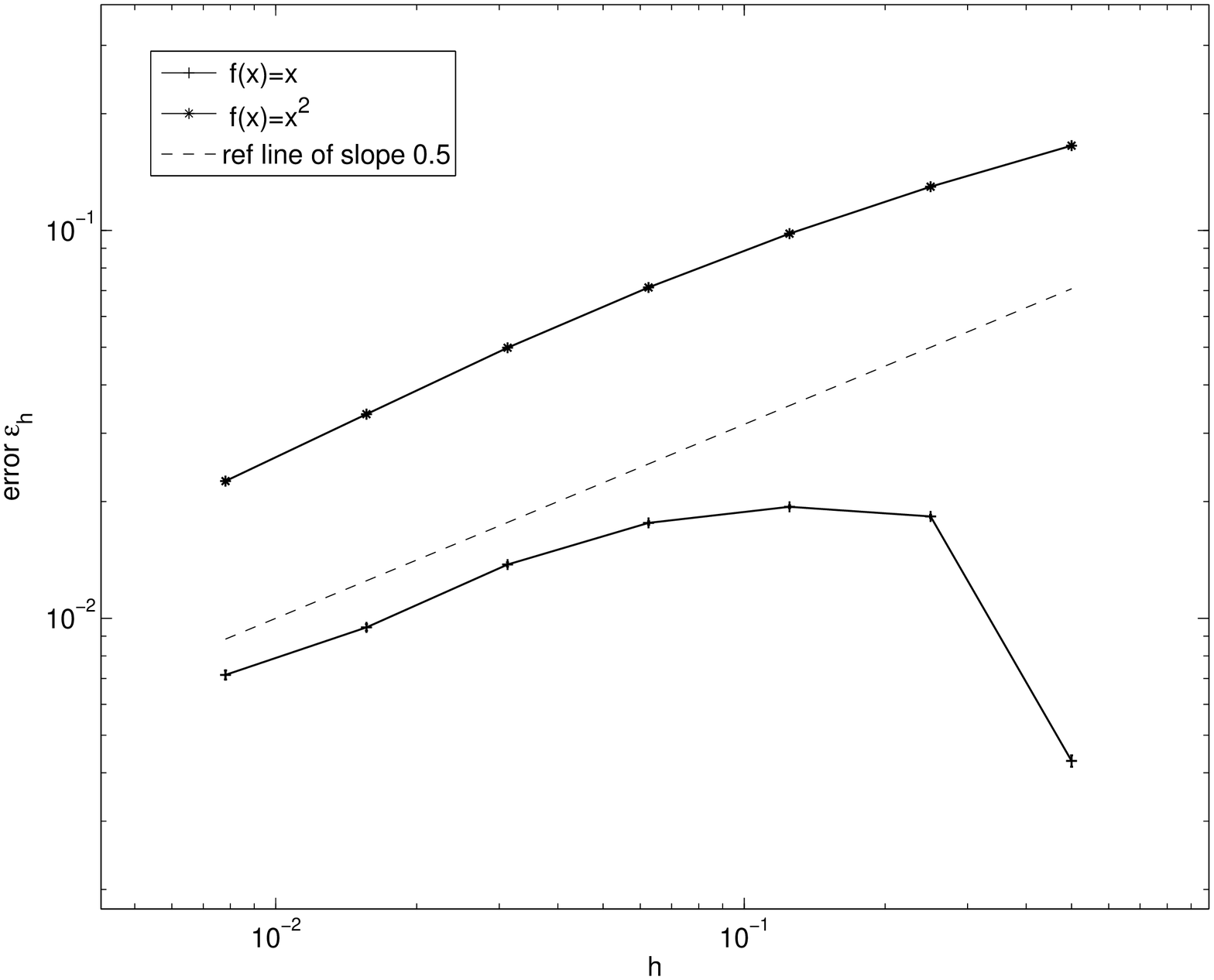}\\
\caption{The weak error of the scheme for Example 1 with test function $f(x)=x$ and $f(x)=x^2$.  The error bars in this plot are smaller than the symbol size. }
\label{sinbt}
\end{figure}

\subsection{Example 2: SDE with smooth coefficients. }
Here we consider the case with smooth diffusion coefficient $D=\sin(x)+2$ and uniform equilibrium distribution $\rho_{eq}=1$. Using \eqref{sde1}, this gives the SDE
\begin{eqnarray*}
dX(t)=\cos(X)dt+\sqrt{4+2\sin(x)}dB
\end{eqnarray*}
with initial condition $X(0)=0$. Here, $\rho_{eq}$ is not normalizable, therefore we do not have a probability density at equilibrium. However, computationally, since we only simulate to finite time, we can still look at the probability distribution of $X(T)$ and its expectation and  moments are well defined. For this SDE, since we do not have the exact solution, we measure the error by subtracting the results from time step length $h/2$ from $h$, i.e.
\begin{equation}
\epsilon_{h}=\abs{\mathbb{E}(f(X_h(T)))-\mathbb{E}(f(X_{h/2}(T)))}
\label{errrel}
\end{equation}
The expectation is approximated by the average over $M=\times 10^7$ trajectories. Figure \ref{sinx+2} shows the error plot compared with the error from Euler-Maruyama (EM) scheme. The EM method shows the expected weak accuracy of order 1. Our method shows the weak accuary of order $1/2$ for the test function $f(x)=x^2$. Furthermore, we observe super-convergence with apparent order 1 for test function $f(x)=x$. A closer look at the leading $\sqrt{h}$ term in the error shows that its coefficient in this case is comparably smaller than the next term due to the effect of $f(x)$ being odd. Therefore, when $h$ is not small enough, the error is dominated by the order $h$ term.

\begin{figure}
\includegraphics[width=1\textwidth]{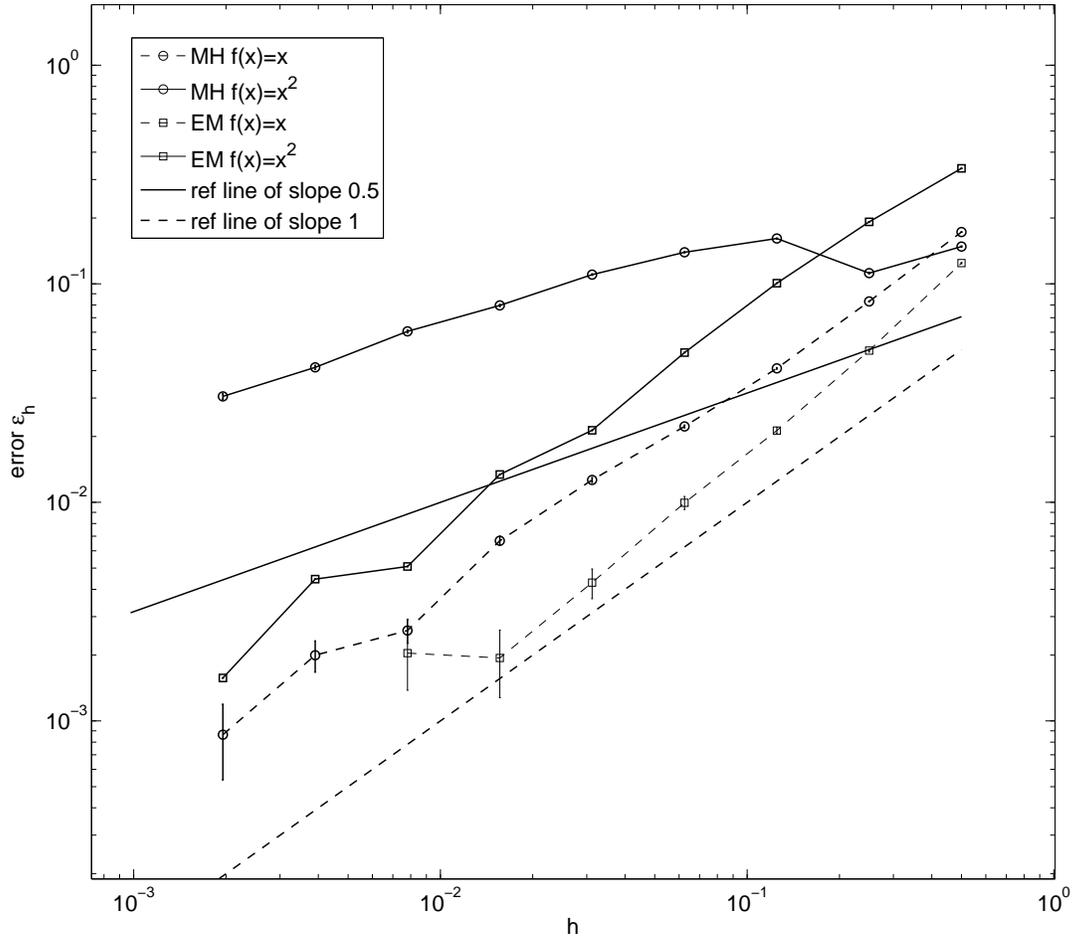}\\
\caption{The weak error of the numerical schemes: Euler-Maruyama (EM) and our scheme (MH) in Example 2 with test functions $f(x)=x,x^2$}
\label{sinx+2}
\end{figure}

\subsection{Example 3: Geometric Brownian Motion. }
For this example, we test our scheme on geometric brownian motion
\begin{eqnarray*}
dX(t)=aXdt+bXdB
\end{eqnarray*}
with $a=1$, $b=1$ are constants. The initial condition is $X_0=1$. We have the exact solution
\begin{equation*}
X(t)=X_0\exp\left( \left(a-\frac{b^2}{2} \right)t+bB(t)\right)=X_0\exp\left(\frac{1}{2}t+B(t)\right)
\end{equation*}
with expectation
\begin{equation*}
\mathbb{E}(X(t))=X_0\exp(t).
\end{equation*}
Firstly we need to rewrite the equation in the form of \eqref{sde1}. Notice that even though geometric brownian motion does not have an equilibrium density, we can still formally let
\begin{eqnarray*}
D=\frac{1}{2}X^2,\ \ \rho_{eq}=1
\end{eqnarray*}
to get the same form of SDE as we want.
Figure \ref{GBM} shows the weak error with test function $f(x)=x$ at time $t=T=1$ compared with the error from the Euler-Maruyama scheme.
The error is measured over  $M=5\times 10^5$ trajectories, using \eqref{errexact} and \eqref{errrel}.
\begin{figure}
\includegraphics[width=1\textwidth]{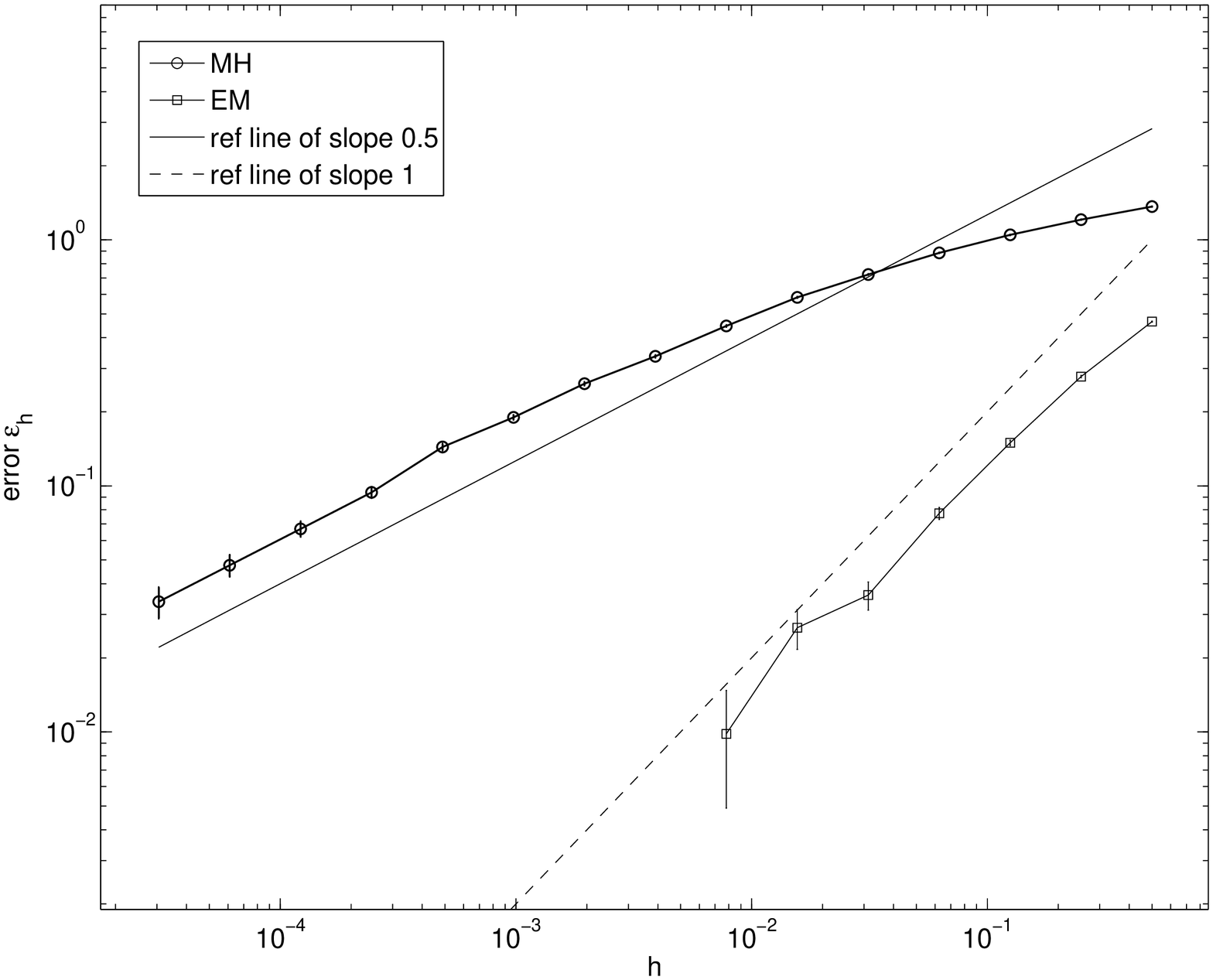}\\
\caption{The weak error of the numerical schemes: Euler-Maruyama (EM) and our scheme (MH) in Example 3 with test functions $f(x)=x$. Error is measured using  \eqref{errexact}.
}
\label{GBM}
\end{figure}
As a result, though geometric brownian motion does not satisfy the conditions in Theorem \ref{mainthm}, the numerical simulation still demonstrates that we can expect convergence in this case with weak accuracy of order $1/2$.

\subsection{Example 4: SDE with piecewise constant diffusion coefficient and equilibrium density.}

Here we study an SDE with equilibrium density $\rho_{eq}=0.5,\  -1<x<1$ and piecewise constant diffusion coefficient.
\begin{eqnarray*}
D(x)=\left\{
\begin{array}{cc}
2, & 1>x\geq0, \\
1, & -1<x<0.\\
\end{array}
\right.
\end{eqnarray*}
In \cite{paulxin}, we showed that our method keeps the correct diffusion coefficient and the exact equilibrium density with this equation. Here we demonstrate the weak convergence. The weak error in this example is calculated using formula
\[
\epsilon_{h}=\left| \mathbb{E}(f(X_h(T)))-\int_{x \in \mathbb{R}}f(x) \rho(x,T) dx \right|
\]
where $\rho(x,t)$ solves the corresponding Fokker-Plank equation,
\[
\frac{\partial \rho_{eq}(x,t)}{\partial t}=\frac{\partial}{\partial x} \left(D(x) \frac{\partial}{\partial x} \rho_{eq}(x,t)\right)
\]
with homogeneous Neumann boundary conditions $\frac{\partial}{\partial x} \rho_{eq}(x,t)=0$ at $x=\pm 1$ and initial condition $\rho(x,0)=\delta(x)$ where $\delta(x)$ is the delta distribution. This divergence form PDE is solved numerically using Crank-Nicolson(CN) scheme with a very fine mesh.
The expectation is approximated by averaging over $M=4\times 10^7$ trajectories. Figure \ref{discon} shows the convegence of the method with test functions $f(x)=x$ and $f(x)=x^2$. In each case we see order $1/2$ convergence despite the discontinuity of $D$ at $x=0$.

\begin{figure}
\includegraphics[width=1\textwidth]{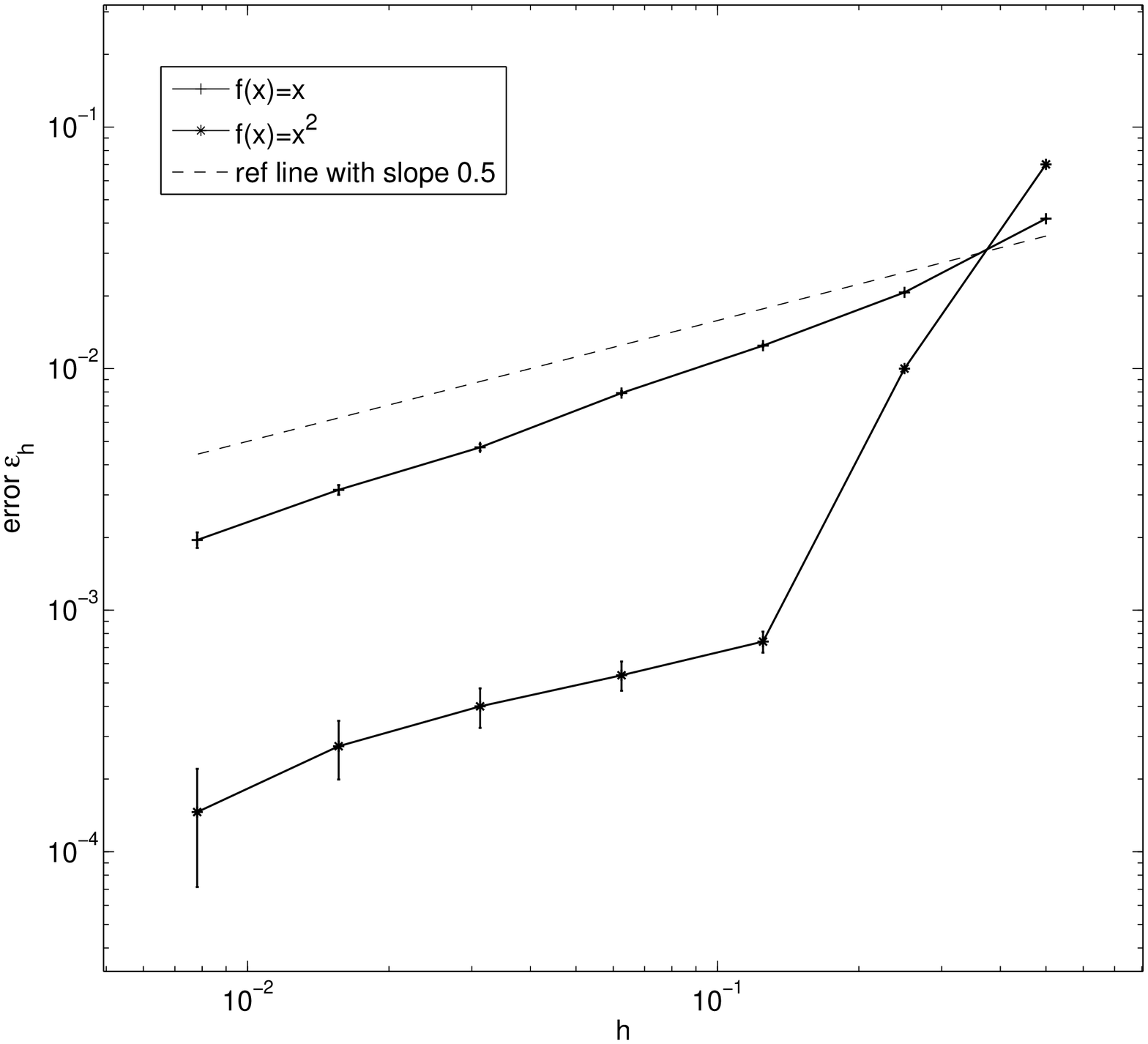}\\
\caption{The weak error of the numerical scheme in Example 4 with test functions $f(x)=x,x^2$.}
\label{discon}
\end{figure}

\bibliographystyle{agsm}	
\bibliography{myrefs}		
\end{document}